\documentclass[12pt]{article}
\usepackage{natbib}
\usepackage{verbatim, color}
\usepackage{amsmath,amsthm, amssymb}
\usepackage{url}

\usepackage{booktabs}
\usepackage{graphicx}
\usepackage{verbatim}
\usepackage{balance}
\usepackage{tikz, pgfplots}
\usepackage[autolanguage]{numprint}

\newtheorem{theorem}{Theorem}
\newtheorem{lemma}[theorem]{Lemma}
\newtheorem{cor}[theorem]{Corollary}

\newtheorem{remark}{Remark}

\newcommand{\R}{\mathbb{R}}

\newcommand{\x}{\mathbf{x}}

\newcommand{\y}{\mathbf{y}}
\renewcommand{\u}{\mathbf{u}}
\renewcommand{\v}{\mathbf{v}}
\newcommand{\w}{\mathbf{w}}

\renewcommand{\H}{\mathcal{H}}

\def\ya{-2.341266903873789}
\def\yb{1.5048776259025467}
\def\ca{3.564567562538916}
\def\cb{8.108261964835723}
\def\Pa{1.6179775280898874}
\def\Pb{1.6179775280898874}
\def\sp{-0.8092982533048366}
\def\val{-0.045636016791851475}
\def\spb{-1.9817233219729211}
\def\valb{23.374359727108025}
\def\ga{1.51}
\def\xmin{-2.2}
\def\xmax{.5}
\nprounddigits{2}

\title{%
Minimax Adaptive Estimation for Finite Sets of Linear Systems
}

\author{Olle Kjellqvist and Anders Rantzer
\footnote{We are grateful to our colleague Dr. Carolina Bergeling for comments that greatly improved the manuscript. This project has received funding from the European Research Council (ERC) under the European Union's Horizon 2020 research and innovation programme under grant agreement No 834142 (ScalableControl). Both authors are with the Department of Automatic Control, Lund University, Lund, Sweden.%
}
}

\begin{document}

\maketitle
\thispagestyle{empty}
\pagestyle{empty}

\begin{abstract}
	For linear time-invariant systems with uncertain parameters belonging to a finite set, we present a purely deterministic approach to multiple-model estimation and propose an algorithm based on the minimax criterion using constrained quadratic programming. The estimator tends to learn the dynamics of the system, and once the uncertain parameters have been sufficiently estimated, the estimator behaves like a standard Kalman filter. 
\end{abstract}

\section{INTRODUCTION}
	\subsection{Problem Statement}
	In this article, we consider output prediction for linear systems of the form
\begin{equation}
        \begin{aligned}
        \label{eq:dynamics}
        x_{t+1} & = Fx_t + Gu_t + w_t \\
        y_t & = Hx_t + v_t, \qquad 0 \leq t \leq N-1, \\
        \end{aligned}
\end{equation}
where $x_t\in \R^n$, $u_t\in \R^p$ and $y_t\in \R^m$ are the states and the measured input and output at time-step $t$, respectively. $w_t\in\R^n$ and $v_t\in\R^m$ are unmeasured process disturbance and measurement noise. The model, $(F, H, G)$ is fixed but unknown, belonging to some finite set 
\[
	\{(F_1, H_1, G_1), \cdots ,(F_K, H_K, G_K)\}.
\]
consiting of of triplets of real-valued matrices. In particular, we are interested in strictly causal estimation of $y_N$, such that the gain from disturbance trajectories $(w_t, v_t)_{t=0}^{N-1}$ to pointwise estimation error $(y_N - Hx_N)$ in some weigthed $\ell_2$-norm is bounded by a constant $\gamma_N > 0$. This means that given positive definite matrices $P_0 \in \R^{n \times n }$, $R \in \R^{m\times m}$ and $Q\in \R^{n\times n}$ and a nominal value of the initial state, $\hat x_0, $ 
\begin{equation}
	\frac{|\hat y_N - Hx_N|^2}{|x_0 - \hat x_0|^2_{P_0^{-1}} + \sum_{t=0}^{N-1}\left(|w_t|_{Q^{-1}}^2 + |v_t|_{R^{-1}}^2\right)} \leq \gamma_N^2,
	\label{eq:opnorm}
\end{equation}
should hold for all disturbances and models compatible with the measurement history $(y_t, u_t)_{t=0}^{N-1}$. This approach is different from the Bayesian approach to filtering where one takes the conditional expectation as the estimate $\hat y_N$.
The interest in worst-case gain is motivated by robust feedback-control from estimates. 
In such settings instability or lack of performance due to model errors is a larger concern than robustness to outliers.
\subsection{Background}

	Simultaneous estimation of states and parameters in linear systems is a bilinear estimation problem. The Maximum-likelihood approach leads to estimates which cannot be put in recursive form and must be obtained by iteration~\cite{Bar1972}. A recursive method can be obtained by parametrizing the dynamical equations and the observer and learning the parameters using the \emph{sequential prediction error} approach. Alternatively, one can augment the state vector with the uncertain parameters and apply nonlinear filtering methods such as the Extended Kalman filter~\cite{Goodwin1984}. Unfortunately, optimality guarantees for such methods are difficult to obtain. One exception is when the system can be modeled as a finite set of linear systems and the noise is Gaussian, then the Maximum-likelihood estimates can be put on a recursive form~\cite{Crassidis2011}.

Solutions based on the multiple-model approach have been tremendously successful in modeling and estimating complex engineering systems. In essence, it consists of two parts: 1) design simpler models for a finite set of possible operating regimes. 2) Run a filter for each model and cleverly combine the estimates. Multiple-model adaptive estimation has been around since the '60s \cite{Magill1965, Lainiotis1971} and has been an active research field since. The estimation approach easily extends to systems where the active model can switch (hybrid systems) by matching a Kalman filter with each possible trajectory. In that case, the number of filters will grow exponentially, which has sparked research into more efficient methods. Notable numerically tractable and suboptimal algorithms for estimation in hybrid systems are the Generalized Pseudo Bayesian~\cite{Ackerson1970, Chang1978}, and the Interacting Multiple Model~\cite{Blom1988}. The algorithms have been coupled with extended and unscented Kalman filters to deal with non-linear systems~\cite{Akca2019}, and \cite{Xiong2015} studied robustness to identification error. In \cite{Guo2008}, the authors pointed out that methods based on Kalman filters are sensitive to noise distributions and proposed an Interactive Multiple Model algorithm based on particle filters to handle non-Gaussian noise at the expense of a 100 fold increase in computation. Recently, machine-learning approaches to classification have been combined with the Interacting Multiple Model estimator \cite{Li2021, Deng2020} and showed improved accuracy in simulations.

The Bayesian approach to the \emph{Multiple-model} estimation problem involves assigning probability distributions to disturbances $(w_t, v_t)$ and models $(F, G, H)$. The estimate is taken as the expected value of $y_N$ conditioned on past measurements. If the disturbances are zero-mean and Gaussian, then the conditional expectation can be computed as the weighted average of Kalman filter estimates (one for each model), weighted by the conditional probability that its model is active.

It is evident in practice that the estimator's performance depends on the quality of the model set. The models must be distinguishable using measured signals, and the models should accurately describe the operating regimes. Since the estimates can be susceptible to non-Gaussian noise, it is surprising that deterministic approaches similar to those studied by the control community in the '80s and '90s have gathered little attention. Recent progress to minimax adaptive control of linear systems with uncertain parameters belonging to a finite set \cite{Rantzer2021} under the assumption of \emph{perfect measurements} has inspired this research into compatible estimation techniques.

\subsection{Contribution}
	In this paper, we formulate the multiple-model estimation problem as a deterministic, two-player dynamic game. In particular, this formulation allows for online computation of the worst-case gain from disturbances to estimation error and tractable synthesis of suboptimal estimators that minimize the worst-case gain. Deterministic dynamic games have played a key role in solving and understanding $\H_\infty$ filtering~\cite{Shen97, Basar95}; our goal in this work has been to take a first step towards extending the advantages of that framework to the multiple model setting. 

\subsection{Outline}
The outline is as follows: First, we introduce notation in Section~\ref{sec:notation}, then we introduce minimax multiple-model filtering and the main results in Section~\ref{sec:minimax}. In Section~\ref{sec:limit}, we present a simplified form for time-invariant systems. We illustrate the theory through a numerical example in Section~\ref{sec:example}. Section~\ref{sec:conclusions} contains concluding remarks, and supporting lemmata are given in the Appendix.
\section{NOTATION}
\label{sec:notation}
The set of $n \times m$-dimensional matrices with real coefficients is denoted $\R^{n\times m}$. The transpose of a matrix $A$ is denoted $A^\top$. For a symmetric matrix $A \in \R^{n\times n}$, we write $A \succ (\succeq) 0$ to say that $A$ is positive (semi)definite. Given $x\in \R^n$ and $A\in \R^{n\times n}$, $|x|^2_A := x^\top A x$. For a vector $x_t \in \R^n$ we denote the sequence of such vectors up to time $t$ by $\x^t := (x_k)_{k=0}^{t}$.
\section{MINIMAX MULTIPLE MODEL FILTERING}
\label{sec:minimax}
	In contrast to the Bayesian approach, our approach is fully deterministic; similarly to \cite{Shen97, Basar95}, we do not make explicit assumptions on the distribution of the noise trajectories $\w^t$ and $\v^t$. We will instead construct a two-player dynamic game between a minimizing player that chooses the estimate, and a maximizing player that chooses dynamics and disturbances. Recall that we are interested in characterizing an estimator $\hat y_N$ such that the gain from disturbances to the pointwise estimation error is bounded by $\gamma_N$. I.e., \eqref{eq:opnorm} holds for all disturbances consistent with \eqref{eq:dynamics} and the data $(\y^{N-1}, \u^{N-1})$. Since the disturbances are unknown, we cannot evaluate \eqref{eq:opnorm} directly. However, define
\begin{multline}
	J_N(\y^{N-1}, \u^{N-1}, \hat y_N) := \sup_{x_0, \w^{N-1}, \v^{N-1}, (F, G, H)}\Bigg\{ |\hat y_N - Hx_N|^2 \\
	- \gamma_N^2\left(|x_0 - \hat x_0|^2_{P_0^{-1}} + \sum_{t=0}^{N-1} \bigg(|w_t|^2_{Q^{-1}} + |v_t|^2_{R^{-1}}\bigg)\right)\Bigg\},
        \label{eq:cost_dist}
\end{multline}
	where the maximization is performed subject to the constraints \eqref{eq:dynamics}. Then \eqref{eq:opnorm} holds if and only if
	\[
	J_N(y^{N-1}, u^{N-1}, \hat y_N) \leq 0.
	\]
In this setting, $w_t = x_{t+1} - Fx_t - Gu_t$ and $v_t = y_t - Hx_t$ are uniquely determined by the states, the measurements and the active model. Inserting into \eqref{eq:cost_dist}, we get
\begin{multline}
	J_N(\y^{N-1}, \u^{N-1}, \hat y_N)= \sup_{\x^{N}, (F,G,H)}\Bigg\{ |\hat y_N - Hx_N|^2 -\gamma_N^2|x_0 - \hat x_0|^2_{P_0^{-1}} \\
	-\gamma_N^2 \sum_{t=0}^{N-1}\bigg ( |x_{t+1} - Fx_t - Gu_t|^2_{Q^{-1}} + |y_t - Hx_t|^2_{R^{-1}}\bigg )\Bigg\}.
	\label{eq:cost}
\end{multline}

	We will call an estimator $\hat y_N^\star$ a minimax estimator if
\begin{multline}
	\inf_{\hat y_N}J_N(y^{N-1}, u^{N-1}, \hat y_N) = J_N(y^{N-1}, u^{N-1}, \hat y_N^\star)=: J_N^\star(y^{N-1}, u^{N-1}),
	\label{eq:minimaxcost}
\end{multline}
holds, where $\hat y_N$ are functions of past data $\y^{N-1}$ and $u^{N-1}$. This constitutes a two-player dynamic game and would be linear quadratic if not for the model being chosen by the maximizing player.
The intuition behind \eqref{eq:minimaxcost} makes sense in the following way. The minimizing player is penalized for deviating from the true (noiseless) output, and the maximizing player is penalized for selecting a model which requires large disturbances $w$ and $v$ to be compatible with the data. As $N$ increases, the penalty for selecting a model different from the truth grows too large, resulting in a learning mechanism. 
It turns out that the cost associated with the disturbance trajectories required to explain each model corresponds to the accumulated prediction errors from a corresponding Kalman filter and that the minimax estimate is a weighted interpolation between the Kalman filter estimates.
\begin{theorem}
        \label{thm:MME}
	Consider matrices $F_1,\ldots,F_K \in \R^{n\times n}$, $H_1,\ldots,H_K \in \R^{m\times n}$, $G_1,\ldots, G_K \in \R^{n\times p}$  and positive definite $Q, P_0\in \R^{n\times n},\ R \in \R^{m\times m}$. Define $P_{t,i}$ according to 
	\[
	\begin{split}
		& P_{0,i} = P_0\\
		& P_{t+1,i} = Q + F_i(P_{t, i}- P_{t,i}H_i^\top(R + H_iP_{t,i}H_i^\top)^{-1}H_i P_{t,i})F_i^\top,
	\end{split}
\]
	and assume that $H_iP_{N,i}H_i^\top\prec \gamma_N^2I$. Then the cost \eqref{eq:cost} is equivalent to
        \begin{multline}
		J_N(\y^{N-1}, u^{N-1}, \hat y_N) =  \max_i \left\{|\hat y_N - H_i\breve x_{N,i}|^2_{(I-\gamma_N^{-2}H_iP_{N,i}H_i^\top)^{-1}} - \gamma_N^2 c_{N,i}\right\}.
		\label{eq:minimax}
        \end{multline}
	$\breve x_{N,i}$ is the Kalman filter estimate of $x_N$ using the $i$th model, and $c_{N,i}$ are generated according to
        \begin{align*}
                \breve x_{0,i}       & = x_0  \\
		\breve{x}_{t+1,i}   & = F_i \breve{x}_{t,i} + K_{t,i}(y_t - H_i\breve{x}_{t,i}) + G_iu_t\\
                K_{t,i}     & = F_iP_{t,i}H_i^\top(R + H_i P_{t,i} H_i^\top)^{-1} \\
                c_{0,i} & = 0 \\
                c_{t+1,i} & = |H_i\breve x_{t,i} -y_t|^2_{(R + H_iP_{t,i}H_i^\top)^{-1}} + c_{t,i}.
        \end{align*}

\end{theorem}
\begin{proof}
		We will perform the maximization over state-trajectories in \eqref{eq:cost} in two steps. First over past trajectories ($\x^{N-1}$) and then over the future state $x_N$\footnote{ $\max_{\x^N}\{\ldots\} = \max_{x_N}\left \{\max_{\x^{N-1}}\{\ldots\}\right\}$.}. The right-hand side of \eqref{eq:cost} becomes 
	\begin{multline*}
		\sup_{x_N, i}\Bigg\{ |\hat y_N - H_ix_N|^2 -\gamma_N^2\inf_{\x^{N-1}}\bigg\{|x_0 - \hat x_0|^2_{P_0^{-1}} \\
		+ \sum_{t=0}^{N-1}\big (|x_{t+1} - F_ix_t - G_iu_t|^2_{Q^{-1}} + |y_t - H_ix_t|^2_{R^{-1}}\big )\bigg\}\Bigg\},
	\end{multline*}
	where $i = 1,\ldots K$ is an index for the active model $(F_i, H_i, G_i)$. Apply Lemma~\ref{lemma:mvn_1} to get
        \begin{multline}
		J_N(y^{N-1}, u^{N-1}, \hat y_N) =\sup_{x_N,i} \left\{|\hat y_N - H_ix_N|^2 - \gamma_N^2 V_{N,i}((x_N, y^{N-1})\right\}\nonumber \\
		 = \sup_{i,x_N} \left \{|\hat y_N - H_ix_N|^2 - \gamma_N^2\left(|x_N - \breve x_N|_{P^{-1}_{N,i}}^2 + c_{N,i}\right)\right\}.
                \label{eq:fcn_of_V}
        \end{multline}
	For fix $\hat y_N$ and $i$, the assumption $H_iP_{N,i}H_i^\top \prec \gamma_N^2I$ guarantees that we maximize a concave function of $x_N$ and we apply Lemma~\ref{lemma:qp1} with $A = H_i,\ X = I,\ Y = P_{N,i}$ to conclude\footnote{The maximizing argument is given by $x_N^\star (\hat y_N, i) = (H_i^\top H_i- \gamma_N^2P^{-1}_{N,i})^{-1}(H_i^\top\hat y_N - P^{-1}_{N,i}\gamma_N^2 \breve x_{N,i})$},
        \begin{equation*}
		J_N(y^{N-1}, u^{N-1}, \hat y_N)	= \max_i |\hat y_N - H_i\breve x_{N,i}|^2_{(I-\gamma_N^{-2}H_iP_{N,i}H_i^\top)^{-1}} - \gamma_N^2 c_{N,i}.
        \end{equation*}
\end{proof}

\begin{remark}
	Theorem~\ref{thm:MME} holds also for time-varying systems, if $F_i$ and $H_i$ are replaced  by $F_{t,i}$ and $H_{t,i}$. Further, $P_0$, $Q$ and $R$ can be time-varying and differ between models.
\end{remark}
	\begin{remark}
		Equation \eqref{eq:minimax} is monotonically increasing in $\gamma_N$ and the smallest $\gamma^\star_N$ such that $J_N(y^{N-1}, u^{N-1}, \hat y_N) \leq 0$ can be found efficiently through bisection.
	\end{remark}
	The below Corollary follows from Theorem~\ref{thm:MME} and describes how to compute the minimax estimator as a convex quadratic program.
	\begin{cor}
		With assumptions as in Theorem~\ref{thm:MME}, consider the convex program 
		\[
			\begin{aligned}
				\underset{\hat y_N, t}{\text{minimize}}\quad & t \\
				\text{subject to:} \quad & |\hat y_N - H_i\breve x_{N,i}|^2_{(I - \gamma_N^{-2}H_iP_{N,i}H_i^\top)^{-1}} - \gamma_N^2 c_{N,i} \leq t \\
				& \forall i = 1\ldots K.
			\end{aligned}
		\]
	The minimizing argument $\hat y_N^\star$ satisfies \eqref{eq:minimaxcost}.
		\label{cor:MME}
	\end{cor}
\begin{remark}
        If the model set is a singleton, then $\hat y^\star_N = Hx_N^\star = H\breve x_N$ is the estimate generated by the Kalman filter, which is a well known result~\cite{Basar95}.
\end{remark}
\subsection{On $c_{N,i}$ and the relation to conditional probability.}

It is known (see for instance~\cite{Crassidis2011}) that if $w_t$ and $v_t$ are uncorrelated Gaussian white noise with covariances $Q$ and $R$, the conditional probability that the measured output $\y^N$ has been generated by the model $(F_i, G_i, H_i)$ and the input $\u^N$ can be expressed as
\[
	p(i|\y^N, \u^N) = \frac{\alpha_N e^{-|y_N - H_i\breve x_{N,i}|^2_{\tilde R_{N,i}}}}{\det(2\pi\tilde R_{N,i})^{1/2}} p(i|\y^{N-1}, \u^{N-1}).
\]
$\alpha_N$ is some normalization constant independent of $i$, and
\[
	\tilde R_{N,i} = R + H_iP_{N,i}H_i^\top,
\]
with $P_{N,i}$ as in Theorem~\ref{thm:MME}. Taking $c_{N,i}$ as in Theorem~\ref{thm:MME} we see that the conditional probability is proportional to $e^{-c_{N+1, i}}$,
\[
	p(i|\y^{N-1}, \u^{N-1}) \propto e^{-c_{N+1,i}}\prod_{t=1}^N \det(2\pi\tilde R_{t,i})^{-1/2}.
\]

\section{STATIONARY SOLUTION}
\label{sec:limit}
For a set of time-invariant systems, we summarize a simple version of the filter in the below theorem.
\begin{theorem}
        \label{thm:limit}
	Consider matrices $F_1,\ldots,F_K \in \R^{n\times n}$, $H_1,\ldots,H_K \in \R^{m\times n}$ and positive definite $Q, P_0\in \R^{n\times n},\ R \in \R^{m\times m}$. Assume that the algebraic Riccati equations
        \begin{equation*}
			P_i  = Q + F_i(P_i  - P_iH_i^\top(R + H_iP_iH_i^\top)^{-1}H_i P_i)F_i^\top,
        \end{equation*}
        have solutions $H_iP_iH_i^\top\prec \gamma_N^2I$. Then a minimax strategy $\hat y^\star_N$ for the game defined by
\begin{multline*}
\min_{\hat y_N}\max_{\x^N, i}\Bigg\{ |\hat y_N - H_ix_N|^2 -\gamma_N^2|x_0 - \hat x_0|^2_{P_i^{-1}} \\
	-\gamma_N^2 \sum_{t=0}^{N-1}\bigg ( |x_{t+1} - F_ix_t - G_iu_t|^2_{Q^{-1}} + |y_t - H_ix_t|^2_{R^{-1}}\bigg )\Bigg\},
\end{multline*}
	and \eqref{eq:dynamics}, is the minimizing argument of 
	\[
		\min_{\hat y_N} \max_i \left\{|\hat y_N - H_i\breve x_{N,i}|^2_{(I-\gamma_N^{-2}H_iP_iH_i^\top)^{-1}} - \gamma_N^2 c_{N,i}\right\}.
	\]
	$\breve x_{N,i}$ is the Kalman filter estimate of $x_N$ using the $i$th model, and $c_{N,i}$ are generated according to
        \begin{align*}
                \breve x_{0,i}       & = x_0  \\
		\breve{x}_{t+1,i}   & = F_i \breve{x}_{t,i} + K_i(y_t - H_i\breve{x}_{t,i})+ G_iu_t\\
                K_i     & = F_iP_iH_i^\top(R + H_i P_i H_i^\top)^{-1} \\
                c_{0,i} & = 0 \\
                c_{t+1,i} & = |H_i\breve x_{t,i} -y_t|^2_{(R + H_iP_iH_i^\top)^{-1}} + c_{t,i}.
        \end{align*}
\end{theorem}
\begin{proof}
	This is a special case of Theorem~\ref{thm:MME}, by replacing $P_0$ with $P_i$.
\end{proof}

\section{EXAMPLE}
\label{sec:example}
In this example, we compare a minimax estimator synthesized using Corollary~\ref{cor:MME}, bisecting over $\gamma_N$, to find the estimator $\hat y_N^\star$ such that \eqref{eq:opnorm} is satisfied for the smallest possible $\gamma_N$. We compare this to a Bayesian multiple-model estimator~\cite{Crassidis2011} and calculate the corresponding bound $\gamma_N$ using Theorem~\ref{thm:MME} and bisection. Consider the uncertain linear system
\[
	\begin{aligned}
		x_{t+1} & = Fx_t + w_t \\
		y_t & = x_t + v_t \\
	\end{aligned}, 
		\quad F  \in \{-1, 1\}.
\]
The weights in \eqref{eq:opnorm} are chosen to be $Q = R = P_0 = 1$. We generate data $\y^{N-1}$ by simulating the system with $F = 1$ and $w_t$, $v_t$ as independent Gaussian white noise with intensity $1$. For $N=5$ we find
\[
	P_{5,1} = P_{5, -1} = \numprint{\Pa}, \\
\]
\[
	\breve x_{5,1} = \numprint{\ya}, \quad \breve x_{5,-1} = \numprint{\yb},
\]
\[
	c_{5,1} = \numprint{\ca}, \quad c_{5,-1} = \numprint{\cb}.
\]

In Fig.~\ref{fig:optplot}, we illustrate \eqref{eq:minimax} for $N=5$ and the estimates. Note that $\gamma = \numprint{\ga}$ can be guaranteed for the minimax estimator, but not the Bayesian. Fig.~\ref{fig:gamma} contains a comparison between the smallest $\gamma_N$ so that \eqref{eq:opnorm} can be guaranteed for the minimax estimator and the Bayesian estimator when $N = 1\ldots 20$.
\begin{figure}
\vspace{1em}
	\centering
	\begin{tikzpicture}[
		declare function={
			Jfun(\x,\y,\c,\P,\g)=(\x-\y)^2/(1-\g^(-2)*\P)-(\g)^2*\c;
		}]
	\begin{axis}[%
			xmin = \xmin, 
			xmax = \xmax, 
			samples = 100, 
			xlabel = $\hat y_5$,
			ylabel = $J_5$,
			scale only axis,
			width = 7cm,
			height = 4cm,
			legend style={legend cell align=left, at={(0.7, 1)}, anchor=north}
			]
		\addplot[gray, dotted, ultra thick, domain = \xmin:\xmax] (\x, {Jfun(\x, \ya, \ca, \Pa, \ga)});
		\addlegendentry{$J^+_5$}
		\addplot[gray,dashed, ultra thick, domain = \xmin:\xmax] (\x, {Jfun(\x, \yb, \cb, \Pb, \ga)});
		\addlegendentry{$J^-_5$}
		\addplot[black, ultra thick, domain = \xmin:\xmax] (\x, {max(Jfun(\x, \ya, \ca, \Pa, \ga),Jfun(\x, \yb, \cb, \Pb, \ga)) });
		\addlegendentry{$J_5 = \max\{J^+_5, J^-_5\}$}
		\addplot[black, ultra thick, domain = \xmin:\xmax] (\x, {max(Jfun(\x, \ya, \ca, \Pa, \ga),Jfun(\x, \yb, \cb, \Pb, \ga)) });
		\addplot[black, dashed, thick, domain = \xmin:\xmax] {0};
		\draw[-stealth,thick, shorten >=5pt] (axis cs:\sp, -8) node[below, blue]{Minimax, $\gamma_5 = \ga$} -- (axis cs:\sp,\val);
		\draw[-stealth,thick, shorten >=5pt, green!40!black] (axis cs:-1.5, 20) node[above]{Bayesian} -- (axis cs:\spb,\valb);
		\addplot[mark=o, ultra thick, mark size = 4pt, blue] coordinates{(\sp, \val)};
		\addplot[mark=x, ultra thick, mark size = 4pt, green!40!black] coordinates{(\spb, \valb)};
	\end{axis}
\end{tikzpicture}%
	\caption{Illustration of the optimization problem \eqref{eq:minimax} for $N=5$, together with the minimax solution and the one given by a Bayesian multiple model estimator for $\gamma_N = \ga$. The minimax estimate has a guaranteed worst-case gain bound from disturbances to observer error lower than $1.51$, whereas the Bayesian estimator does not. Here $J^+_5 =|\hat y_5-\breve x_{5, 1}|^2_{(I- \gamma_5^{-2}P_{5,1})^{-1}} - c_{5,1}$ corresponds to $F = 1$, whereas $J^-_5$ (defined similarly) corresponds to $F = -1$. $J_5 = J_5(\y^5, 0, \hat y_5)$ is then equivalent to \eqref{eq:minimax}.}
	\label{fig:optplot}
\end{figure}
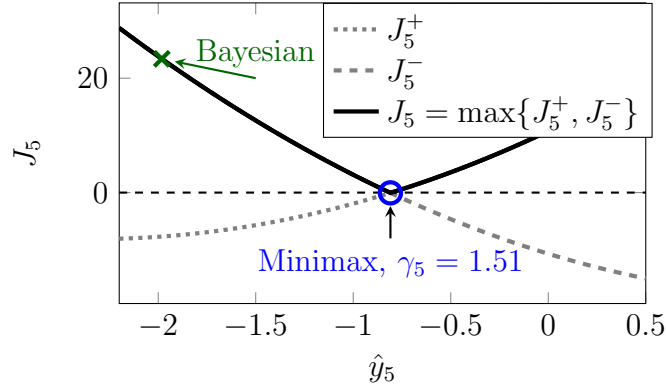

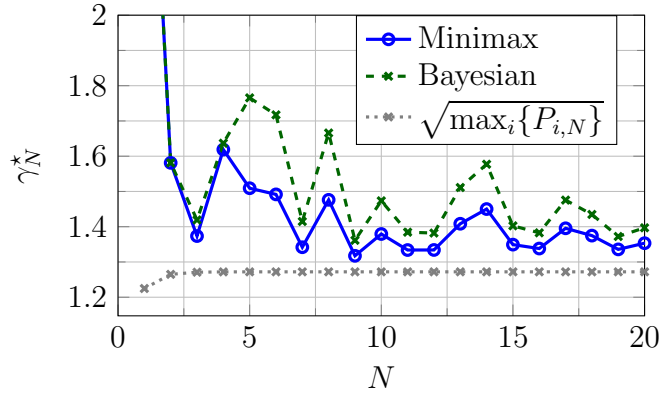
\begin{figure}
	\centering
	\pgfplotstableread[col sep=semicolon]{figures/gammasF1.csv}{\datatable}
\begin{tikzpicture}
                        \begin{axis}[%
                                xlabel={$N$},
                                ylabel={$\gamma_N^\star$},
                                scale only axis,
                                xmin=0,
                                xmax=20,
				ymax=2.0,
                                width=7cm,
                                height=4cm,
                                grid=both,
                                minor tick num=1,
                                legend style={legend cell align=left, at={(0.7, 1)}, anchor=north}]
                                        \addplot[mark=o, very thick, blue] table[x=t, y=gammamini] from \datatable;
					\addlegendentry{Minimax}
					\addplot[mark=x, mark options=solid, very thick, dashed, green!40!black] table[x=t, y=gammabayes] from \datatable;
                                        \addlegendentry{Bayesian}
					\addplot[mark=x, mark options=solid, very thick, dotted, gray] table[x=t, y=gammamin] from \datatable;
					\addlegendentry{$\sqrt{\max_i\{P_{i,N}\}}$}

\end{axis}%
\end{tikzpicture}%
~                 
	\caption{The smallest $\gamma_N$ such that $J_N(\y^{N-1}, 0, \hat y_N) \leq 0$ for the minimax estimator (blue) compared to the Bayesian multiple-model adaptive estimator (green) for one realization.}
	\label{fig:gamma}
\end{figure}
\section{CONCLUSIONS}
\label{sec:conclusions}
We stated the minimax criterion for output prediction, where the dynamics belong to a finite set of linear systems and proposed a minimax estimation strategy. The strategy can be implemented as a convex program, and the resulting estimate is a weighted interpolation of Kalman filter estimates. We showed in a numerical example how to apply the theoretical results to compute the worst-case gain from disturbances to error for any multi-model estimation algorithm online and how to generate estimates that minimize the said gain.

By running a minimax estimator in parallel to another estimator, we can measure the worst-case performance level of the other estimator.
A large difference in performance levels indicates that the nominal estimator may be highly sensitive to errors in the noise model.

Predetermining the smallest achievable gain from disturbances to estimation errors is still an open research problem, that is, finding necessary and sufficient conditions such that
	\[
		\sup_{\y^{N-1}} J_N^\star(\y^{N-1}, \u^{N-1}) \leq 0.
	\]
In future work, we plan to develop a Multiple-model adaptive estimator with a prescribed $\ell_2$-gain bound from disturbance to error and methods for infinite sets of linear systems.

\section*{APPENDIX --- SUPPORTING LEMMATA}
        \begin{lemma}
        The cost function
        \label{lemma:mvn_1}
        \begin{multline}
                V_{N,i}(x_N, \y^{N-1}) = \min_{\x^{N-1}}\Bigg\{|x_0 - \hat x_0|_{P_0^{-1}}^2 \\
		+ \sum_{k=1}^{N-1}(|x_{t+1} - F_ix_t - G_i u_t|_{Q^{-1}}^2 + |y_t - H_ix_t|_{R^{-1}}^2)\Bigg\}
        \label{eq:first_cost_1}
\end{multline}
        under the dynamics \eqref{eq:dynamics}, is of the form
        \[
		V_{t,i}(x, \y^{t-1}) = |x - \breve x_{t,i}|^2_{P_{t,i}} + c_{t,i},
        \]
		where $P_{t,i}$ and $c_{t,i}$ are generated as
        \begin{equation*}
                \begin{aligned}
                P_{0,i} & = P_0\\
                        P_{t+1,i} & = Q + F_iP_{t, i}F_i^\top \\ 
			& \quad - F_iP_{t,i}H_i^\top(R + H_iP_{t,i}H_i^\top)^{-1}H_i P_{t,i}F_i^\top \\
                \breve x_{0,i}       & = x_0  \\
			\breve{x}_{t+1,i}   & = F_i \breve{x}_{t,i} + K_{t,i}(y_t - H_i\breve{x}_{t,i}) +G_iu_t\\
                K_{t,i}     & = F_iP_{t,i}H_i^\top(R + H_i P_{t,i} H_i^\top)^{-1} \\
                c_{0,i} & = 0 \\
                c_{t+1,i} & = |H_i\breve x_{t,i} -y_t|^2_{(R + H_iP_{t,i}H_i^\top)^{-1}} + c_{t,i}.
        \end{aligned}
	\end{equation*}
        \end{lemma}
\begin{proof}
	The proof builds on forward dynamic programming~\cite{Cox64}, and is similar to one given in~\cite{Goodwin05} but differ in the assumption that $F_i$ is not invertible. Further, the constant terms $c_{t,i}$ are explicitly computed. 
	The cost function $V_N$\footnote{We relax the index $i$ in this proof} can be computed recursively
\begin{align}
	V_1(x, \y^0)    & = |x-x_0|_{P_0^{-1}}^2 \label{eq:first_v0}\\
	V_{t+1}(x, \y^t)    & = \min_{\xi}|x-F\xi -Gu_t|_{Q^{-1}}^2 \nonumber \\
				& \qquad + |y_t-H\xi|_{R^{-1}}^2 + V_t(\xi,\y^{t-1}). \label{eq:first_vk}
\end{align}

With a slight abuse of notation, we assume a solution of the form $V_t(x) = |x - \breve{x}_t|_{P_t^{-1}} + c_t$ and solve for the minimum
        \begin{multline*}
		V_{t+1}(x) = \min_\xi |x - Gu_t|^2_{Q^{-1}} + |\xi|^2_{F^\top Q^{-1} F + H^\top R^{-1} H + P_t^{-1}} \\
		- 2(F^\top Q^{-1}(x-Gu_t) + H^\top R^{-1}y_t + P_t^{-1}\breve{x}_t)^\top\xi + |y_t|_{R^{-1}}^2 + |\breve{x}|_{P_t^{-1}}.
        \end{multline*}
        Assume at this stage $S_t := F^\top Q^{-1}F + H^\top R^{-1}H + P_t^{-1} \succ 0$, then the minimizing $\xi^\star$ is a stationary point
\[
	\xi^\star = S_t^{-1}(F^\top Q^{-1}(x-Gu_t) + H^\top R^{-1}y_t + P_t^{-1}\breve{x}_t)
\]
and the resulting partial cost
\begin{multline}
	|x-\breve{x}_{t+1}|^2_{P_{t+1}^{-1}} + c_{t+1} = |x-Gu_t|_{Q^{-1}}^2 + |y_t|_{R^{-1}}^2 + |\breve{x}_t|^2_{P_t^{-1}} \\
	 - |F^\top Q^{-1}(x-Gu_t) + H^\top R^{-1}y_t + P_t^{-1}\breve x_t|_{S_t^{-1}}^2 + c_t.
        \label{eq:PX}
\end{multline}
	Since this should hold for arbitrary $x$ and 
	\[
		x - \breve x_{t+1} = (x - Gu_t) - (\breve x_{t+1}-Gu_t),
	\]
	we get
\[
        \begin{aligned}
                P_{t+1}^{-1}    & = Q^{-1} - Q^{-1}FS_t^{-1}F^\top Q^{-1} \\
		\breve{x}_{t+1} - Gu_t   & = P_{t+1}Q^{-1}FS_t^{-1}(H^\top R^{-1}y_t + P_t^{-1} \breve x_t)
        \end{aligned}
\]
The expression for calculating $P_{t+1}$ can be further simplified using the Woodbury identity,
        \[
                \begin{aligned}
                        P_{t+1}^{-1}    & = (Q + F(H^\top R^{-1} H + P_t^{-1})^{-1} F^\top)^{-1}\\
                        P_{t+1}         & = Q + FP_tF^\top - FP_tH^\top(R + HP_t H^\top)^{-1}H P_t F^\top,
                \end{aligned}
        \]
        where we used the Woodbury matrix identity twice.
        Inserting these expressions into \eqref{eq:PX}, applying the Woodbury matrix identity to $S_t^{-1}F^\top (Q - FS_t^{-1}F^\top)^{-1}S_t^{-1}+ S_t^{-1} =(S_t - F^\top Q^{-1}F)^{-1} = (H^\top R^{-1}H + P_t^{-1})^{-1}$ gives 

        \begin{equation*}
                \begin{aligned}
                        c_{t+1} & = -|H^\top R^{-1}y_t + P_t^{-1}\breve{x}_t|_{(H^\top R^{-1}H + P_t^{-1})^{-1}}^2 +|y_t|_{R^{-1}}^2 + |\breve{x}_t|_{P_t^{-1}}^2 + c_t \\
                                & = |H\hat x_t - y_t|^2_{(R+HP_tH^\top)^{-1}} + c_t
                \end{aligned}
        \end{equation*}

        Next we show that $\breve x$ can be formulated as a state-observer
        \begin{align*}
		\breve x_{t+1} -Gu_t   & = P_{t+1}Q^{-1}FS_t^{-1}(H^\top R^{-1}y_t + P_t^{-1}\breve x) \\
                                & = P_{t+1}Q^{-1}FS_t^{-1}H^\top R^{-1}(y_t - H\breve x_t) \\
				&  \qquad + P_{t+1}Q^{-1}FS_t^{-1}(H^\top R^{-1}H + P_t^{-1})\breve x_t \\
	\end{align*}
                   Use the matrix inversion lemma $(A + BCD)^{-1}BC = A^{-1}B(C + DA^{-1}B)^{-1}$.
		   \begin{align*}
			  P_{t+1}Q^{-1}FS_t^{-1} & = -(-Q^{-1} + Q^{-1}FS_t^{-1}F^\top Q^{-1})^{-1}Q^{-1}FS_t^{-1} \\
                &  = -(-Q^{-1})^{-1}(Q^{-1}F)(S_t - F^\top Q^{-1}F)^{-1} \\
                &  = F(H^\top R^{-1}H + P_t^{-1})^{-1}.
        \end{align*}

        Insert in to the previous expression and conclude
        \[
		\breve x_{t+1}  = F\breve x_t + K_t(y_t - H\breve{x}) +Gu_t,
        \]
        where
        \[
                K_t = FP_t H^\top(R + HP_t H^\top)^{-1}
        \]
\end{proof}
\balance
\begin{lemma}
	\label{lemma:qp1}
	For $x \in \R^n$, $v$, $y \in \R^m$, a non-zero matrix $A\in \R^{n\times m}$, positive-definite matrices $X\in R^{n\times n}$ and $Y \in \R^{m^\times m}$, and a positive real number $\gamma_N > 0$ such that
	\[
		A^\top X^{-1}A - \gamma_N^2 Y^{-1} \prec 0,
	\]
	it holds that
	\begin{multline}
		\max_v\left\{ |x-Av|_{X^{-1}}^2 - \gamma_N^2|y - v|^2_{Y^{-1}}\right\} \\
		= |x - Ay|^2_{(X - \gamma_N^{-2}AYA^\top)^{-1}}.
		\label{eq:qp}
	\end{multline}
\end{lemma}
\begin{proof}
	Expanding the left-hand side of \eqref{eq:qp} and equating the gradient with $0$ we get
	\begin{align*}
		&\max_v\left\{ |x-Av|_{X^{-1}}^2 - \gamma_N^2|y - v|^2_{Y^{-1}}\right\} \\
		&= \max_v\Big\{|v|^2_{A^\top X^{-1} A - \gamma_N^2Y} + |x|^2_{X^{-1}} - \gamma_N^2 |y|^2_{Y^{-1}}
		 - 2v^\top(A^\top X^{-1}x - \gamma_N^2 Y^{-1})y \Big\} \\
		&= |x|^2_{X^{-1}} - \gamma_N^2|y|^2_{Y^{-1}}
		- |A^\top X^{-1}x - \gamma_N^2Y^{-1}y|_{(A^\top X^{-1}A - \gamma_N^2Y^{-1})^{-1}} \\
		&= |x|^2_{X^{-1} - X^{-1}A^\top(A^\top X^{-1}A - \gamma_N^2 Y^{-1})^{-1}A^\top X^{-1}} \\
		&\qquad +|y|^2_{-\gamma_N^2 Y^{-1} - \gamma_N^2 Y^{-1}(A^\top X^{-1}A - \gamma_N^2Y^{-1})^{-1}Y^{-1}\gamma_N^2}\\
		&\qquad -2x^\top X^{-1}A(A^\top X^{-1}A - \gamma_N^2 Y^{-1})^{-1}(-\gamma_N^2Y^{-1})y\\
		&= |x|^2_{(X - \gamma_N^{-2}AYA^\top)^{-1}}+ |Ay|^2_{(X - \gamma_N^{-2}AYA^\top)^{-1}}
		- 2x^\top(X - \gamma_N^{-2}AYA^\top)^{-1}Ay \\
		&= |x - Ay|^2_{(X - \gamma_N^{-2}AYA^\top)^{-1}}.
	\end{align*}
\end{proof}

\bibliographystyle{plainnat}
\bibliography{references}

\begin{thebibliography}{18}
\providecommand{\natexlab}[1]{#1}
\providecommand{\url}[1]{\texttt{#1}}
\expandafter\ifx\csname urlstyle\endcsname\relax
  \providecommand{\doi}[1]{doi: #1}\else
  \providecommand{\doi}{doi: \begingroup \urlstyle{rm}\Url}\fi

\bibitem[{Ackerson} and {Fu}(1970)]{Ackerson1970}
G.~{Ackerson} and K.~{Fu}.
\newblock On state estimation in switching environments.
\newblock \emph{IEEE Transactions on Automatic Control}, 15\penalty0
  (1):\penalty0 10--17, 1970.
\newblock \doi{10.1109/TAC.1970.1099359}.

\bibitem[Akca and Önder Efe(2019)]{Akca2019}
Alper Akca and M.~Önder Efe.
\newblock Multiple model kalman and particle filters and applications: A
  survey.
\newblock \emph{IFAC-PapersOnLine}, 52\penalty0 (3):\penalty0 73--78, 2019.
\newblock ISSN 2405-8963.
\newblock \doi{https://doi.org/10.1016/j.ifacol.2019.06.013}.
\newblock URL
  \url{https://www.sciencedirect.com/science/article/pii/S2405896319300977}.
\newblock 15th IFAC Symposium on Large Scale Complex Systems LSS 2019.

\bibitem[Bar-Shalom(1972)]{Bar1972}
Y.~Bar-Shalom.
\newblock Optimal simultaneous state estimation and parameter identification in
  linear discrete-time systems.
\newblock \emph{IEEE Transactions on Automatic Control}, 17\penalty0
  (3):\penalty0 308--319, 1972.
\newblock \doi{10.1109/TAC.1972.1100005}.

\bibitem[{Basar} and {Bernhard}(1995)]{Basar95}
T.~{Basar} and P.~{Bernhard}.
\newblock \emph{{$H_\infty$}-Optimal Control and Related Minimax Design
  Problems --- A dynamic Game Approach}.
\newblock Birkhauser, 1995.

\bibitem[{Blom} and {Bar-Shalom}(1988)]{Blom1988}
H.~A.~P. {Blom} and Y.~{Bar-Shalom}.
\newblock The interacting multiple model algorithm for systems with markovian
  switching coefficients.
\newblock \emph{IEEE Transactions on Automatic Control}, 33\penalty0
  (8):\penalty0 780--783, 1988.
\newblock \doi{10.1109/9.1299}.

\bibitem[{Chang} and {Athans}(1978)]{Chang1978}
C.~B. {Chang} and M.~{Athans}.
\newblock State estimation for discrete systems with switching parameters.
\newblock \emph{IEEE Transactions on Aerospace and Electronic Systems},
  AES-14\penalty0 (3):\penalty0 418--425, 1978.
\newblock \doi{10.1109/TAES.1978.308603}.

\bibitem[{Cox}(1964)]{Cox64}
H.~{Cox}.
\newblock On the estimation of state variables and parameters for noisy dynamic
  systems.
\newblock \emph{IEEE Transactions on Automatic Control}, 9\penalty0
  (1):\penalty0 5--12, 1964.

\bibitem[Crassidis and Junkins(2011)]{Crassidis2011}
John~L. Crassidis and John~L. Junkins.
\newblock \emph{Optimal Estimation of Dynamic Systems, Second Edition (Chapman
  \& Hall/CRC Applied Mathematics \& Nonlinear Science)}.
\newblock Chapman \& Hall/CRC, 2nd edition, 2011.
\newblock ISBN 1439839859.

\bibitem[Deng et~al.(2020)Deng, Li, and Li]{Deng2020}
Lichuan Deng, Da~Li, and Ruifang Li.
\newblock Improved {IMM} algorithm based on {RNNs}.
\newblock \emph{Journal of Physics: Conference Series}, 1518:\penalty0 012055,
  apr 2020.
\newblock \doi{10.1088/1742-6596/1518/1/012055}.
\newblock URL \url{https://doi.org/10.1088/1742-6596/1518/1/012055}.

\bibitem[{Goodwin} et~al.(2005){Goodwin}, {De Dona}, and {Seron}]{Goodwin05}
G.~C. {Goodwin}, J.~A. {De Dona}, and M.~M. {Seron}.
\newblock \emph{Constrained Control and Estimation --- An Optimization
  Approach}.
\newblock Springer-Verlag, 2005.

\bibitem[Goodwin and Sin(1984)]{Goodwin1984}
Graham~Clifford Goodwin and Kwai~Sang Sin.
\newblock \emph{Adaptive filtering prediction and control / Graham C. Goodwin
  and Kwai Sang Sin.}
\newblock Prentice-Hall information and system sciences series. Prentice-Hall,
  Englewood Cliffs, N.J., 1984.
\newblock ISBN 013004069X.

\bibitem[{Lainiotis}(1976)]{Lainiotis1971}
D.~G. {Lainiotis}.
\newblock Partitioning: A unifying framework for adaptive systems, i:
  Estimation.
\newblock \emph{Proceedings of the IEEE}, 64\penalty0 (8):\penalty0 1126--1143,
  1976.
\newblock \doi{10.1109/PROC.1976.10284}.

\bibitem[Li et~al.(2021)Li, Zhang, and Li]{Li2021}
Da~Li, Pei Zhang, and Ruifang Li.
\newblock Improved {IMM} algorithm based on {XGBoost}.
\newblock \emph{Journal of Physics: Conference Series}, 1748:\penalty0 032017,
  jan 2021.
\newblock \doi{10.1088/1742-6596/1748/3/032017}.
\newblock URL \url{https://doi.org/10.1088/1742-6596/1748/3/032017}.

\bibitem[{Magill}(1965)]{Magill1965}
D.~{Magill}.
\newblock Optimal adaptive estimation of sampled stochastic processes.
\newblock \emph{IEEE Transactions on Automatic Control}, 10\penalty0
  (4):\penalty0 434--439, 1965.
\newblock \doi{10.1109/TAC.1965.1098191}.

\bibitem[Rantzer(2021)]{Rantzer2021}
Anders Rantzer.
\newblock Minimax adaptive control for a finite set of linear systems, 2021.

\bibitem[Ronghua et~al.(2008)Ronghua, Zheng, Xiangnan, and Junliang]{Guo2008}
Guo Ronghua, Qin Zheng, Li~Xiangnan, and Chen Junliang.
\newblock Interacting multiple model particle-type filtering approaches to
  ground target tracking.
\newblock \emph{Journal of Computers}, 3, 07 2008.
\newblock \doi{10.4304/jcp.3.7.23-30}.

\bibitem[{Shen} and {Deng}(1997)]{Shen97}
X.~{Shen} and L.~{Deng}.
\newblock Game theory approach to discrete {$H_\infty$} filter design.
\newblock \emph{IEEE Transactions on Signal Processing}, 45\penalty0
  (4):\penalty0 1092--1095, 1997.

\bibitem[Xiong et~al.(2015)Xiong, Wei, and Liu]{Xiong2015}
K.~Xiong, C.L. Wei, and L.D. Liu.
\newblock Robust multiple model adaptive estimation for spacecraft autonomous
  navigation.
\newblock \emph{Aerospace Science and Technology}, 42:\penalty0 249--258, 2015.
\newblock ISSN 1270-9638.
\newblock \doi{https://doi.org/10.1016/j.ast.2015.01.021}.
\newblock URL
  \url{https://www.sciencedirect.com/science/article/pii/S1270963815000371}.

\end{thebibliography}

\end{document}